\documentclass[10pt,leqno]{amsart}
\usepackage[hyphens]{url}
\usepackage[hypertexnames=false]{hyperref}
\usepackage[hyphenbreaks]{breakurl}

\usepackage[dvipsnames]{xcolor}
\usepackage{indentfirst,amsmath,amsthm,amssymb,amsfonts,setspace,graphicx,mathtools,array,etoolbox,algorithm,algorithmicx,algpseudocode,tikz,ytableau,changepage,cleveref,thmtools,thm-restate}
\usepackage[enableskew]{youngtab}
\usepackage[most]{tcolorbox}
\usepackage{amsaddr}

\usepackage[
backend=biber,
style=alphabetic,
]{biblatex}
\DeclareFieldFormat{extraalpha}{#1}
\DeclareLabelalphaTemplate{
	\labelelement{
		\field[final]{shorthand}
		\field{label}
		\field[strwidth=3,strside=left,ifnames=1]{labelname}
		\field[strwidth=1,strside=left]{labelname}
	}
}

\newtheorem{mytheorem}{Theorem}

\newtheorem{theorem}{Theorem}
\newtheorem{lemma}{Lemma}[section]

\newtheorem{proposition}[lemma]{Proposition}

\newtheorem{conjecture}[lemma]{Conjecture}
\theoremstyle{definition}
\newtheorem{definition}[lemma]{Definition}

\theoremstyle{remark}
\newtheorem*{rmk}{Remark}
\newtheorem*{notation}{Notation}

\hypersetup{pdftitle={Computations of Stable Multiplicites},colorlinks=true, linkcolor=blue, filecolor=black, urlcolor=blue, citecolor=Orange }

\DeclareFontFamily{OMS}{rsfs}{\skewchar\font'60}
\DeclareFontShape{OMS}{rsfs}{m}{n}{<-5>rsfs5 <5-7>rsfs7 <7->rsfs10 }{}
\DeclareSymbolFont{rsfs}{OMS}{rsfs}{m}{n}
\DeclareSymbolFontAlphabet{\scr}{rsfs}

\newcommand{\la}{\lambda}

\newcommand{\si}{\sigma}

\newcommand{\ze}{\zeta}

\makeatletter
\newcommand{\colim@}[2]{%
	\vtop{\m@th\ialign{##\cr
			\hfil$#1\operator@font colim$\hfil\cr
			\noalign{\nointerlineskip\kern1.5\ex@}#2\cr
			\noalign{\nointerlineskip\kern-\ex@}\cr}}%
}
\newcommand{\colim}{%
	\mathop{\mathpalette\colim@{\rightarrowfill@\textstyle}}\nmlimits@
}
\makeatother

\newcommand{\C}{\mathbb {C}}
\newcommand{\N}{\mathbb {N}}
\newcommand{\R}{\mathbb {R}}

\newcommand{\Z}{\mathbb {Z}}
\newcommand{\Q}{\mathbb {Q}}
\newcommand{\F}{\mathbb {F}}

\DeclareMathOperator{\id}{{id}}
\DeclareMathOperator{\Id}{{Id}}

\DeclareMathOperator{\sgn}{{sgn}}

\newcommand{\pconf}{\operatorname{PConf}}
\newcommand{\conf}{\operatorname{Conf}}

\DeclareMathOperator{\Ind}{{Ind}}
\DeclareMathOperator{\Res}{{Res}}

\title{Computations of Stable Multiplicities in the Cohomology of Configuration Space} 
\author{Emil Geisler}
\address{University of California Los Angeles}
\email{emilg@math.ucla.edu}

\begin{document}

	\begin{abstract}
		We describe an algorithm to compute the stable multiplicity of a family of irreducible representations in the cohomology of ordered configuration space of the plane. Using this algorithm, we compute the stable multiplicities of all families of irreducibles given by Young diagrams with $23$ boxes or less up to cohomological degree $50$. In particular, this determines the stable cohomology in cohomological degrees $0 \leq i \leq 11$. We prove related qualitative results and formulate some conjectures. 
	\end{abstract}

	\maketitle
	\bigskip
	\tableofcontents
	
	\section{Introduction}
	The $n$th ordered configuration space of the plane is the space of $n$-tuples in $\C^n$ with distinct coordinates, denoted in this article as $\pconf_n(\C)$:
	$$\pconf_n(\C) := \{(x_1, \dots, x_n) \in \C^n \ | \ x_i \neq x_j \text{ when $i \neq j$}\} \subset \C^n.$$ 
	There is a natural action of the symmetric group of $n$ elements (denoted here by $S_n$) on $\pconf_n(\C)$ by permuting the coordinates. This induces an action of $S_n$ on cohomology so $H^i(\pconf_n(\C);\C)$ is an $S_n$-representation. 
	\vspace{.1in}\par
	The study of the cohomology of $\pconf_n(\C)$ has had a rich history. A full historical account would be out of place here, but we give notable developments to contextualize our results. Arnol'd described the cohomology rings $H^\bullet(\pconf_n(\C);\Z)$ and the Betti numbers $H^i(\pconf_n(\C);\Z)$ in 1969 \cite{Arnold1969}, and Brieskorn expanded this work to include descriptions of the cohomology of more general hyperplane complements in 1971 \cite{brieskorn2006groupes}. 
	Fred Cohen related the cohomology of $\pconf_n(\C)$ to the Gerstenhaber operad and described the ``little disks'' operad in the 1970s \cite{cohen2006homology}. A connection between the cohomology of $\pconf_n(\C)$ (and related spaces) to combinatorics of partitions was described by Orlik-Solomon in 1980 \cite{orlik1980combinatorics}, leading to work by Bj\"{o}rner, Stanley, and many others \cite{bjorner1994subspace, stanley1982some}. 
	Lehrer-Solomon gave a formula for the $S_n$-representation theoretic structure of $H^i(\pconf_n(\C);\Z)$ in 1986 \cite{lehrersolomon} which we use to prove Theorem \ref{boundvanishing}. In the 1990s, Getzler studied the mixed Hodge structure of configuration spaces \cite{getzler1996mixedhodgestructuresconfiguration}. 
	In particular, this enabled unstable computations related to the stable ones performed here (see Section \ref{past_data}). In 2013, Church-Farb found that $H^i(\pconf_n(\C))$ \textit{stabilizes} as $n \to \infty$ when viewed as an $S_n$ representation \cite{church2013representation}. We now describe the notion of Church-Farb's \textit{representation stability}.
	\vspace{.1in}\par
	Let $\la$ be a non-increasing tuple of positive integers $\la = (a_1, \dots, a_r)$ with $k = \sum_i a_i$. Then for any $n$ such that $n - k \geq a_1$, let $V(\la)_n$ denote the irreducible representation of $S_n$ given by the partition $(n - k) + a_1 + \dots + a_r$ of $n$ with respect to Young's description of irreducible representations of $S_n$ (as in e.g. \cite[Sec. 4.1]{fulton_and_harris}). We write $V(\la)$ to denote the family of irreducible representations $\{V(\la)_n\}_{n \geq a_1 + k}$. Recall that any finite dimensional $S_n$ representation can be expressed as a sum of irreducibles. We let $d_{i,n}(\la)$ denote the \textit{multiplicity} of $V(\la)_n$ in $H^i(\pconf_n(\C);\C)$, which is the number of direct summands isomorphic to $V(\la)_n$ that occur when $H^i(\pconf_n(\C);\C)$ is written as a sum of irreducible $S_n$ representations. Representation stability in the sense of Church-Farb \cite{church2013representation} refers to the phenomenon that for any partition $\la$ and degree of cohomology $i$, $d_{i,n}(\la)$ is eventually constant as $n \to \infty$, and is zero for all but finitely many $\la$ \cite[Theorem 4.1]{church2013representation}. In particular, $d_{i,n}(\la)$ is constant for $n \geq 4i$, which was refined to $n \geq 3i + 1$ by Hersh-Reiner  who also showed this bound is sharp \cite{hersh2017representation}.
	Let $d_i(\la)$ denote the value $d_{i,n}(\la)$ for $n \gg 0$. We will refer to $d_i(\la)$ as the \textit{stable multiplicity} of $V(\la)$ in $H^i(\pconf(\C);\C)$.
	\vspace{.1in}\par
	The non-negative integers $d_i(\la)$ have geometric, arithmetic, and combinatorial relevance.  % stolen from chen
	Determining these coefficients is an open question first proposed by Farb \cite[pg. 3]{farbrepstability} \cite[pg. 39]{church2013representation}, although the unstable coefficients $d_{i,n}(\la)$ have been of interest for longer. Ideally, a simple formula exists for $d_i(\la)$ in terms of $i$ and $\la$, although none has been found so far. A satisfactory solution could consist of completing the following objectives:
	\begin{enumerate}
		\item[\textbf{(1)}] Determine a formula for $d_i(\la)$.
		\item[\textbf{(2)}] Compute $d_i(\la)$ for a large number of examples.
		\item[\textbf{(3)}] Prove qualitative properties of $d_i(\la)$, like when $d_i(\la) = 0$ or asymptotics in terms of $i$ and the size of $\la$.
	\end{enumerate}
	
	Chen answered objective \textbf{(1)} by expressing the $d_i(\la)$ as the coefficients of a formal power series in terms of the character polynomial of $V(\la)$ (see \hyperref[polystatformula]{Theorem 2}), expanding on work by Church-Ellenberg-Farb \cite[Theorem 1]{pn-fqt} and Fulman \cite{fulman2014}. Results on objective \textbf{(2)} have been limited to small $\la$ or small $i$: Chen computed $d_i(\la)$ for small $\la$ (when $\la$ is a partition of $2$ or less) \cite[Ex. 1,2,3]{chen2016twisted}, and Farb computed the coefficients $d_i(\la)$ for all $\la$ when $i = 4$ \cite[pg. 9]{farbrepstability}. 
	Bergstr\"{o}m made related computations on the cohomology of the moduli space of $n$-pointed genus zero curves \cite{bergstrom_data} which can be used to recover the stable coefficients $d_i(\la)$ when $i \leq 6$. Further discussion of past computations can be found in Section \ref{past_data}.
	\vspace{.1in}\par
	Based on Chen's formula for $d_i(\la)$, we implement an algorithm  which expands objective \textbf{(2)} significantly. In particular, we compute the values $d_i(\la)$ for all partitions $\la$ of $N$ with $N \leq 23$ and all cohomological degrees $0 \leq i \leq 50$. Combined with Theorem \ref{boundvanishing}, this data determines $d_{i}(\la)$ for \textit{all} partitions $\la$ when $i \leq 11$. 
	Our data has informed conjectures about $d_i(\la)$ (objective \textbf{(3)}) and led to the proofs of Theorems \ref{boundvanishing}, \ref{asymptotic}, and \ref{length_bound} given in Section \ref{proofs}. 
	
	\subsection{Results and Conjectures}
	\begin{notation}
		Fix a positive integer $k$ and let $\la = (a_1, \dots, a_r)$ be a partition of $k$, so $k = a_1 + \dots + a_r$. We will also write $|\la| = k$ to denote that $\la$ is a partition of $k$.
	\end{notation}
	
	\begin{mytheorem} \label{boundvanishing}
		For $0 \leq i < k/2$, $d_i(\la) = 0$. 
	\end{mytheorem}
	
	\begin{mytheorem}
		\label{asymptotic}
		Suppose $k > 0$. The sequence $d_0(\la), d_1(\la), \dots, d_i(\la), \dots$ is asymptotically $Ci^{k-1}$ for $C \in \Q^{>0}$ and thus \emph{eventually} non-decreasing.  More explicitly,
		$$\lim_{i \to \infty} \frac{d_i(\la)}{i^{k-1}} = \frac{2 \dim \la}{(k-1)!},$$ 
		where $\dim \la$ is the dimension of the irreducible $S_k$ representation given by $\la$.
	\end{mytheorem}
	
	\begin{mytheorem}\label{length_bound}
		For $0 \leq i < r$, $d_i(\la) = 0$.
	\end{mytheorem}
	
	\begin{rmk}
		Theorem \ref{boundvanishing} is necessary for our computation of the stable cohomology of $\pconf(\C)$ for a fixed $i$ in Section \ref{fixeddegree}, since it ensures we only need to compute $d_i(\la)$ for $\la$ with $2i$ or less boxes, and the remaining coefficients are zero.
	\end{rmk} 
	
	\begin{conjecture} \label{nondecconjecture}
		So long as $k > 0$ (i.e. $V(\la)$ is not the family of trivial representations), the sequence $d_0(\la), d_1(\la), \dots$ is non-decreasing.
	\end{conjecture}
	
	\begin{conjecture}\label{upperbound}
		If $d_l(\la)$ is the first non-zero term of $d_0(\la), d_1(\la), \dots $, then $l \leq k$. Furthermore, the bound of $l = k$ is achieved if and only if $\la = (1, 1, \dots, 1)$, so $V(\la)$ is the family of $k$th wedge powers of the standard representation. 
	\end{conjecture}
	
	\begin{rmk}
		Conjectures \ref{nondecconjecture} and \ref{upperbound} together would imply that for any nonempty partition $\la$, $V(\la)$ appears stably (i.e., with multiplicity at least $1$) in $H^i(\pconf(\C);\C)$ for all $i \geq k$. 
	\end{rmk}
	
	The data of $d_i(\la)$ for all $0 \leq i \leq 50$ and $|\la| \leq 23$ as well as the stable decompositions of $H^i(\pconf(\C);\C)$ for $0 \leq i \leq 11$ are available on Github  \cite{github} as well as with an interactive web interface \cite{mywebsite}.
	
	\subsection{Outline}
	In Section \ref{methods} we discuss the equations used to compute $d_i(\la)$. In Section \ref{results} we give a subset of the computational results and contextualize with previous computations. In Section \ref{proofs} we give proofs of Theorems \ref{boundvanishing}, \ref{asymptotic} and \ref{length_bound}. In Appendix \ref{charpolysappendix} we give background on character polynomials and in Appendix \ref{alg-details} we provide implementation details of our algorithm and consider its efficiency. 
	
	\subsection{Acknowledgements}
	We thank Sean Howe for extensive guidance, advice, and moral support throughout the research process. We also thank Sean Howe for thorough review of previous versions of the paper. 
	We thank an anonymous reviewer for the proof of Theorem \ref{length_bound} and many helpful comments on the historical background.
	
	\subsection{Disclosure Statement}
	The authors report there are no competing interests to declare.
	
	\section{Methods} \label{methods}
	
	Church-Ellenberg-Farb \cite{pn-fqt} originally described a connection via the Grothendieck-Lefschetz fixed point formula between twisted cohomology of a variety (topology) and weighted polynomial statistics over a finite field (arithmetic). The stability of $S_n$ representations of $\pconf_n(\C)$ reflects the convergence of polynomial statistics of $\conf_n(\F_q)$, the space of degree $n$ square free polynomials in $\F_q[x]$. Expanding on an observation of Fulman that the weighted polynomial statistics in question could be computed using generating functions \cite{fulman2014}, Chen refined this connection to describe an explicit formula for generating functions $\Big((-1)^iz^i t^n d_{i,n}(\la)\Big)_{i, n \geq 0}$ and $\Big((-1)^iz^id_{i}(\la)\Big)_{i \geq 0}$ in terms of the character polynomial of $V(\la)$. While our algorithm here only relies on their coefficients, we note that these power series have arithmetic meaning when $z$ is replaced with $q^{-1}$ for $q$ a prime power. For a more detailed discussion, see the introduction to Chen \cite{chen2016twisted}. 
	\vspace{.1in}\par
	Our algorithm takes as input a partition $\la$ and positive integer $N$ and returns $\{d_i(\la)\}_{0 \leq i \leq N}$ as the coefficients of a formal power series. It is based upon equations by Macdonald and Chen. 
	For implementation details, see Appendix \ref{alg-details}. Note that the character polynomials $\binom{X}{\rho}$ for $\rho$ a partition are defined in Appendix \ref{charpolysappendix}.
	
	\begin{notation}
		For $\mu$ and $\rho$ both partitions of $n$, 
		$\chi^\mu_\rho$ denotes the character of the irreducible representation given by $\mu$ evaluated on $\rho$, treating $\rho$ as a conjugacy class.
	\end{notation}
	
	\begin{restatable}{theorem}{charpolyformula} \textbf{(Macdonald \cite[ex 1.7.14]{macdonald1998symmetric})} \label{charpolyformula}
		For any partition $\la$, there is a unique character polynomial $\chi^\la$ which agrees with the character of every $V(\la)_n \in V(\la)$. It is given by the equation
		\begin{equation} \label{youngtocharpoly}
			\chi^\la = \sum_{|\rho| \leq |\la|} \binom{X}{\rho}  (-1)^{|\la| - |\rho|} \sum_\mu \chi_\rho^\mu,
		\end{equation}
		where $\mu$ is summed over all partitions $|\mu| = |\rho|$ such that $\la - \mu$ is a \emph{vertical strip}, that is, $\la$ can be obtained by adding $|\la| - |\mu|$ boxes to $\mu$ without adding more than one box in any one row.
	\end{restatable}
	
	\begin{notation}
		Let $M_{k}(z):= \frac{1}{k}\sum_{d|k}\mu\big(\frac{k}{d}\big) z^d$ denote the $k$th necklace polynomial. For $f$ a Laurent polynomial, define $\binom{f(z)}{k}:=\frac{f(z)(f(z) - 1) \dots (f(z) - k + 1)}{k!}$.
	\end{notation}
	
	\begin{restatable}{theorem}{polystatformula} \textbf{(Chen \cite[eq 2.11]{chen2016twisted})} \label{polystatformula}
		Let $\rho = 1^{j_1}2^{j_2} \dots r^{j_r}$ be a partition. Define the following formal power series in $z$:
		\begin{equation}\label{polystat} 
			\Phi_\rho^{\infty}(z) := (1 - z)\prod_{i=1}^r \binom{M_{i}(z^{-1})}{j_i} (z^{i} - z^{2i} + \dots )^{j_i}.
		\end{equation}
		For $\la$ a partition, let $\chi^\la$ be the character polynomial of $V(\la)$. Suppose $\chi^\la = \sum_{|\rho| \leq |\la|} F^\la_\rho \binom{X}{\rho}$ for coefficients $F^\la_\rho \in \Q$. Then as a formal power series, we have
		\begin{equation}\label{chen_formal_equation}
			\sum_{i \geq 0}(-1)^i d_i(\la) z^i = \sum_{|\rho| \leq |\la|} F^\la_\rho \Phi_\rho^{\infty}(z).
		\end{equation}
		
	\end{restatable}

	\section{Results} \label{results}
	\subsection{Example Computational Results}
	The following table contains the data of the generating function $(-1)^iz^id_i(\la)$ computed up to $z^{30}$ with $\la$ given by the Young diagram in the left column. In the \texttt{results.csv} file available at \cite{github}, %how can i reference this
	this data is extended to include $d_i(\la)$ for all $\la,i$ with $|\la| \leq 23$ and $0 \leq i \leq 50$, which is almost entirely novel.  
	
	\newcolumntype{Y}[1]{>{\centering\arraybackslash$}m{#1}<{$}} % vertical centered math mode
	
	% gives some more space in the rows
	\renewcommand{\arraystretch}{1.5}
	\begin{center}
		\begin{tabular}{|Y{3cm}|Y{10cm}|}\hline
			\text{Young Diagram} & \text{Formal power series $(-1)^id_i(\la)z^i$ computed up to $z^{30}$} \\
			\hline
			\hline
			\text{Trivial} & 
			1 - z
			\\
			\hline
			\yng(1) & 
			-z^{1} + 2z^{2} - 2z^{3} + 2z^{4} - 2z^{5} + 2z^{6} - 2z^{7} + 2z^{8} - 2z^{9} + 2z^{10} - 2z^{11} + 2z^{12} - 2z^{13} + 2z^{14} - 2z^{15} + 2z^{16} - 2z^{17} + 2z^{18} - 2z^{19} + 2z^{20} - 2z^{21} + 2z^{22} - 2z^{23} + 2z^{24} - 2z^{25} + 2z^{26} - 2z^{27} + 2z^{28} - 2z^{29} + 2z^{30} + \dots
			\\
			\hline
			\yng(2) &
			-z^{1} + 2z^{2} - 3z^{3} + 6z^{4} - 9z^{5} + 10z^{6} - 11z^{7} + 14z^{8} - 17z^{9} + 18z^{10} - 19z^{11} + 22z^{12} - 25z^{13} + 26z^{14} - 27z^{15} + 30z^{16} - 33z^{17} + 34z^{18} - 35z^{19} + 38z^{20} - 41z^{21} + 42z^{22} - 43z^{23} + 46z^{24} - 49z^{25} + 50z^{26} - 51z^{27} + 54z^{28} - 57z^{29} + 58z^{30} + \dots
			\\
			\hline
			\yng(1,1) &
			2z^{2} - 5z^{3} + 6z^{4} - 7z^{5} + 10z^{6} - 13z^{7} + 14z^{8} - 15z^{9} + 18z^{10} - 21z^{11} + 22z^{12} - 23z^{13} + 26z^{14} - 29z^{15} + 30z^{16} - 31z^{17} + 34z^{18} - 37z^{19} + 38z^{20} - 39z^{21} + 42z^{22} - 45z^{23} + 46z^{24} - 47z^{25} + 50z^{26} - 53z^{27} + 54z^{28} - 55z^{29} + 58z^{30} + \dots
			\\
			\hline
			\yng(3) &
			z^{2} - 4z^{3} + 8z^{4} - 14z^{5} + 24z^{6} - 35z^{7} + 46z^{8} - 61z^{9} + 79z^{10} - 97z^{11} + 117z^{12} - 140z^{13} + 165z^{14} - 192z^{15} + 220z^{16} - 250z^{17} + 284z^{18} - 319z^{19} + 354z^{20} - 393z^{21} + 435z^{22} - 477z^{23} + 521z^{24} - 568z^{25} + 617z^{26} - 668z^{27} + 720z^{28} - 774z^{29} + 832z^{30} + \dots
			\\
			\hline
			\yng(2,1) &
			2z^{2} - 7z^{3} + 16z^{4} - 30z^{5} + 47z^{6} - 68z^{7} + 94z^{8} - 123z^{9} + 156z^{10} - 194z^{11} + 235z^{12} - 280z^{13} + 330z^{14} - 383z^{15} + 440z^{16} - 502z^{17} + 567z^{18} - 636z^{19} + 710z^{20} - 787z^{21} + 868z^{22} - 954z^{23} + 1043z^{24} - 1136z^{25} + 1234z^{26} - 1335z^{27} + 1440z^{28} - 1550z^{29} + 1663z^{30} + \dots
			\\
			\hline
		\end{tabular}
	\end{center}

	\subsection{Fixed Degree of Cohomology} \label{fixeddegree}
	An open question in representation stability referenced by Farb \cite[pg. 3]{farbrepstability} is to express the stable cohomology of $H^i(\pconf(M);\C)$ for a manifold $M$ and fixed $i$ as a decomposition of irreducible representations. Here we consider the case $M = \C$. Based on the bound of degree in Theorem \ref{boundvanishing} and our computational results, we compute these decompositions for $0 \leq i \leq 11$. The data of $i \geq 6$ is available through an interactive web interface and as a \texttt{.csv} file at \cite{mywebsite} or with code on GitHub \cite{github}. In the following, let $V(0)$ denote the family of trivial representations.
	
	$$H^0(\pconf(\C);\C) \cong V(0).$$

	$$H^1(\pconf(\C);\C) \cong V(0) \oplus V(1) \oplus V(2).$$

	$$H^2(\pconf(\C);\C) \cong V(1)^{\oplus 2} \oplus V(2)^{\oplus 2} \oplus V(1,1)^{\oplus 2} \oplus V(3) \oplus V(2,1)^{\oplus 2} \oplus V(3,1).$$

	$$H^3(\pconf(\C);\C) \cong V(1)^{\oplus 2} \oplus V(2)^{\oplus 3} \oplus V(1,1)^{\oplus 5} \oplus V(3)^{\oplus 4} \oplus V(2,1)^{\oplus 7} \oplus V(1,1,1)^{\oplus 3}$$ 
	$$\oplus V(4)
	\oplus V(3,1)^{\oplus 6} \oplus V(2,2)^{\oplus 2} \oplus V(2,1,1)^{\oplus 4} \oplus V(4,1)^{\oplus 2} \oplus V(3,2)^{\oplus 2}$$ 
	$$\oplus V(3,1,1)^{\oplus 2} \oplus V(2,2,1) \oplus V(4,1,1) \oplus V(3,3).$$

	$$H^4(\pconf(\C);\C) \cong V(1)^{\oplus 2} \oplus V(2)^{\oplus 6} \oplus V(1,1)^{\oplus 6} \oplus V(3)^{\oplus 8} \oplus V(2,1)^{\oplus 16}$$
	$$ \oplus V(1,1,1)^{\oplus 9} \oplus V(4)^{\oplus 6} \oplus V(3,1)^{\oplus 19} \oplus V(2,2)^{\oplus 12} \oplus V(2,1,1)^{\oplus 17}$$
	$$ \oplus V(1,1,1,1)^{\oplus 5} \oplus V(5)^{\oplus 2} \oplus V(4,1)^{\oplus 12} \oplus V(3,2)^{\oplus 14} \oplus V(3,1,1)^{\oplus 16} $$
	$$ \oplus V(2,2,1)^{\oplus 10} \oplus V(2,1,1,1)^{\oplus 7} \oplus V(5,1)^{\oplus 3} \oplus V(4,2)^{\oplus 7} \oplus V(4,1,1)^{\oplus 8} \oplus V(3,3)^{\oplus 4}$$
	$$ \oplus V(3,2,1)^{\oplus 9} \oplus V(3,1,1,1)^{\oplus 5} \oplus V(2,2,2)^{\oplus 2} \oplus V(2,2,1,1)^{\oplus 2} \oplus V(5,2) \oplus V(5,1,1)^{\oplus 2}$$
	$$ \oplus V(4,3)^{\oplus 2} \oplus V(4,2,1)^{\oplus 3} \oplus V(4,1,1,1)^{\oplus 2} \oplus V(3,3,1)^{\oplus 2} \oplus V(3,2,2) \oplus V(3,2,1,1)$$
	$$ \oplus V(5,1,1,1) \oplus V(4,3,1).$$

	$$H^5(\pconf(\C);\C) \cong V(1)^{\oplus 2} \oplus V(2)^{\oplus 9} \oplus V(1,1)^{\oplus 7} \oplus V(3)^{\oplus 14} \oplus V(2,1)^{\oplus 30}$$
	$$ \oplus V(1,1,1)^{\oplus 15} \oplus V(4)^{\oplus 17} \oplus V(3,1)^{\oplus 46} \oplus V(2,2)^{\oplus 34} \oplus V(2,1,1)^{\oplus 45}$$
	$$ \oplus V(1,1,1,1)^{\oplus 17} \oplus V(5)^{\oplus 10} \oplus V(4,1)^{\oplus 43} \oplus V(3,2)^{\oplus 53} \oplus V(3,1,1)^{\oplus 62} $$
	$$ \oplus V(2,2,1)^{\oplus 47} \oplus V(2,1,1,1)^{\oplus 36} \oplus V(1,1,1,1,1)^{\oplus 7} \oplus V(6)^{\oplus 3} \oplus V(5,1)^{\oplus 22}$$
	$$ \oplus V(4,2)^{\oplus 45} \oplus V(4,1,1)^{\oplus 44} \oplus V(3,3)^{\oplus 20} \oplus V(3,2,1)^{\oplus 66} \oplus V(3,1,1,1)^{\oplus 39}$$
	$$ \oplus V(2,2,2)^{\oplus 18} \oplus V(2,2,1,1)^{\oplus 25} \oplus V(2,1,1,1,1)^{\oplus 12} \oplus V(6,1)^{\oplus 5} \oplus V(5,2)^{\oplus 17} $$
	$$  \oplus V(5,1,1)^{\oplus 19} \oplus V(4,3)^{\oplus 19} \oplus V(4,2,1)^{\oplus 41} \oplus V(4,1,1,1)^{\oplus 23} \oplus V(3,3,1)^{\oplus 23}$$
	$$ \oplus V(3,2,2)^{\oplus 19} \oplus V(3,2,1,1)^{\oplus 28} \oplus V(3,1,1,1,1)^{\oplus 9} \oplus V(2,2,2,1)^{\oplus 7} \oplus V(2,2,1,1,1)^{\oplus 5}$$
	$$ \oplus V(6,2)^{\oplus 3} \oplus V(6,1,1)^{\oplus 3} \oplus V(5,3)^{\oplus 5} \oplus V(5,2,1)^{\oplus 12} \oplus V(5,1,1,1)^{\oplus 9}$$
	$$ \oplus V(4,4)^{\oplus 5} \oplus V(4,3,1)^{\oplus 14} \oplus V(4,2,2)^{\oplus 10} \oplus V(4,2,1,1)^{\oplus 12} \oplus V(4,1,1,1,1)^{\oplus 5} \oplus V(3,3,2)^{\oplus 5}  $$
	$$ \oplus V(3,3,1,1)^{\oplus 8} \oplus V(3,2,2,1)^{\oplus 5} \oplus V(3,2,1,1,1)^{\oplus 3} \oplus V(2,2,2,2) \oplus V(6,2,1) \oplus V(6,1,1,1)^{\oplus 2}$$
	$$ \oplus V(5,4) \oplus V(5,3,1)^{\oplus 3} \oplus V(5,2,2) \oplus V(5,2,1,1)^{\oplus 3} \oplus V(5,1,1,1,1)^{\oplus 2} \oplus V(4,4,1)^{\oplus 2}  $$
	$$ \oplus V(4,3,2)^{\oplus 3} \oplus V(4,3,1,1)^{\oplus 3} \oplus V(4,2,2,1) \oplus V(4,2,1,1,1) \oplus V(3,3,2,1) \oplus V(6,1,1,1,1)$$
	$$\oplus V(5,3,1,1) \oplus V(4,4,2).$$
	
	\subsection{Comparison with Existing Data}\label{past_data}
	
	Most prior computations of the stable coefficients $d_i(\la)$ have been limited to small $i$ or small $\la$: Church-Ellenberg-Farb computed $d_i(\la)$ for $\la = (1)$ \cite[Proposition 4.5]{pn-fqt} and $\la = (1,1)$ \cite[pg. 38]{pn-fqt} as functions of $i$, which Chen reproduced along with $d_i((2))$ \cite[Example 3]{chen2016twisted}. Farb included the stable decomposition of $H^4(\pconf_n(\C);\C)$ as an $S_n$ representation in \cite[pg. 9]{farbrepstability}, thus computing $d_4(\la)$ for all $\la$. 
	\vspace{.1in}\par
	A larger source of prior computations comes from before the notion of representation stability (in the sense of Church-Farb) was introduced.
	Bergstr\"{o}m computed $H^i(\mathcal{M}_{0,n})$ as a sum of irreducible $S_n$ representations, where $\mathcal{M}_{0,n}$ is the moduli space of $n$-pointed genus zero curves \cite{bergstrom_data}, based on formulas of Getzler \cite{getzler1995operads} for $n \leq 22$. See the introduction to \cite{bergstrom2007cohomology} for Bergstr\"{o}m's discussion of these computations and extensions to higher genus moduli spaces. This data is relevant to our computations since there is an $S_n$-equivariant homotopy equivalence \cite[Corollary 3.10]{getzler1995operads}
	$$\pconf_n(\C) \cong S^1 \times \mathcal{M}_{0, n+1}.$$
	Here $S_n$ acts trivially on $S^1$ and acts on $\mathcal{M}_{0,n+1}$ via the inclusion $S_n \hookrightarrow S_{n+1}$ by $(ij) \mapsto ((i+1)(j+1))$. 
	In particular, this implies by the K\"{u}nneth formula we have an isomorphism of $S_n$ representations:
	\begin{equation}
		\label{moduli_space_comparison_with_pconf}
		H^i(\pconf_n(\C)) \cong \Res^{S_{n+1}}_{S_n}H^i(\mathcal{M}_{0, n+1}) \oplus \Res^{S_{n+1}}_{S_n}H^{i-1}(\mathcal{M}_{0,n+1}).
	\end{equation}
	Using the branching rule for restriction of $S_n$ representations \cite[\S 7, Corollary 3.1]{fulton1997young}, we can extract the unstable coefficients $d_{i,n}(\la)$ for $n \leq 21$ from Bergstr\"{o}m's data. 
	Combined with Hersh-Reiner's sharp bound \cite{hersh2017representation} that $d_{i,n}(\la)$ is constant for $n \geq 3i+1$, Bergstr\"{o}m's computations yield the coefficients $d_i(\la)$ for $i \leq 6$ and a partial description of the coefficients $d_7(\la)$. 
	\vspace{.1in}\par
	As part of our work, we compare each of the above sources of data with our own and confirm that they all agree. We also remark that using Chen's \textit{unstable} formulas \cite[Theorem 1]{chen2016twisted}, it would be possible to modify our work to compute the unstable coefficients $d_{i,n}(\la)$, which would offer a more detailed comparison with Bergstr\"{o}m's results. 
	
	\section{Proofs of Theorems} \label{proofs}
	
	\subsection{Bound on Vanishing Multiplicity}
	In this subsection we prove Theorem \ref{boundvanishing}.
	First we prove the following two Lemmas:
	\begin{enumerate}
		\item[\textbf{\ref{dimh_i}.}] There is a polynomial $f_i(n)$ in $n$ of degree $2i$ such that $f_i(n) = \dim H^i(\pconf_n(\C);\C)$ for all integers $n \geq 2i$.
		\item[\textbf{\ref{dimV(la)_n}.}] Let $\la$ be a partition of $k$. $\dim V(\la)_n$ is given by a polynomial $g_\la(n)$ in $n$ of degree $k$ for $n \geq k$. 
	\end{enumerate}
	
	Given these results, suppose that $d_i(\la)$ is non-zero for some $i \in \N$ and $\la$ a partition of $k$. Then by definition, $V(\la)_n$ appears as a direct summand of $H^i(\pconf_n(\C);\C)$ for all $n \gg 0$. In particular, we must have $\dim V(\la)_n \leq \dim H^i(\pconf_n(\C);\C)$ for all $n \gg 0$, so $g_\la(n) \leq f_i(n)$ for all $n \gg 0$. Therefore since both sides are polynomials in $n$ (with positive leading coefficients), we have $\deg g_\la \leq \deg f_i$, so $k \leq 2i$ proving Theorem \ref{boundvanishing}.
	
	\begin{lemma} \label{dimh_i}
		There is a polynomial $f_i(n)$ in $n$ of degree $2i$ such that $f_i(n) = \dim H^i(\pconf_n(\C);\C)$ for all integers $n \geq 2i$.
	\end{lemma}
	
	\begin{rmk}
		Arnol'd determined the Betti number $\dim H^i(\pconf_n(\C);\C)$ to be $s(n, n-i)$, where $s(a,b)$ denotes the $(a,b)$th Stirling number of the first kind \cite{Arnold1969}. One can prove Equation \ref{dim_formula} for $s(n, n-i)$ using elementary methods (for instance as shown here \cite[eq 55]{malenfant2011finiteclosedformexpressionspartition}). We use Lehrer and Solomon's description \cite{lehrersolomon} of $H^i(\pconf_n(\C);\C)$ as an $S_n$ representation to recover Equation \ref{dim_formula} but remark that the use of this sophisticated result is not necessary.
	\end{rmk}
	
	\begin{notation}
		A partition $\la = (a_1, \dots, a_s)$ of $n$ may be written in \textit{cycle notation} $\la = 1^{\la_1} 2^{\la_2} \dots r^{\la_r}$ where $\la_i$ is the number of occurrences of $i$ in $(a_1, \dots, a_s)$. For example the cycle notation of the partition $4+2+2+1 = (4,2,2,1)$ is $1^12^23^04^1$, or just $1^12^24^1$. 
	\end{notation}
	
	\begin{proof}
		Let $P(n)$ denote the set of partitions of $n$ and let $T(i,n) \subset P(n)$ denote the partitions $\mu$ of $n$ with exactly $n-i$ terms:
		\begin{equation*}\label{def_of_T}
			T(i,n) := \bigg\{1^{\mu_1} \dots r^{\mu_r} \ \bigg| \ \sum_j j\mu_j = n, \sum_j \mu_j = n-i\bigg\}.
		\end{equation*}
		Subtracting the linear equations defining $T(i,n)$, equivalently we have
		$$T(i,n) = \bigg\{1^{\mu_1} \dots r^{\mu_r} \ \bigg| \ \sum_j (j-1)\mu_j = i, \mu_1 = n-i - \sum_{j=2}^r \mu_j \bigg\}.$$
		In particular, an element $\mu = 1^{\mu_1} \dots r^{\mu_r}$ of $T(i,n)$ is determined by $\mu_2, \dots, \mu_r$ since $\mu_1$ must satisfy $\mu_1 = n - i - \sum_{j=2}^r \mu_j$. Furthermore, notice that the equation $\sum_j (j-1)\mu_j = i$ means that we have a function 
		$$T(i,n) \xrightarrow{\phi} P(i) \hspace{.2in}1^{\mu_1} \dots r^{\mu_r} \mapsto 1^{\mu_2} \dots (r-1)^{\mu_r}.$$ The function $\phi$ is an injection since $\mu \in T(i,n)$ is determined by $\mu_2, \dots, \mu_r$. Furthermore, $\la = 1^{\la_1} \dots r^{\la_r} \in P(i)$ is in the image of $\phi$ if and only if $n - i - \sum_{j} \la_j \geq 0$. Since $\la \in P(i)$ satisfies
		$$\sum_j \la_j \leq \sum_j j \la_j = i,$$
		we have $n - i - \sum_{j} \la_j \geq n - 2i$ for any $\la \in P(i)$. Therefore,
		the map $\phi:T(i,n) \to P(i)$ is a bijection for $n \geq 2i$.
		\vspace{.1in}\par
		By Lehrer-Solomon's computation of $H^i(\pconf_n(\C);\C)$ as an $S_n$ representation \cite[eq. 1.3]{lehrersolomon},
		$$H^i(\pconf_n(\C);\C) \cong \bigoplus_{\mu \in T(i,n)} \Ind_{Z(c_\mu)}^{S_n} (\xi_\mu),$$
		where $c_\mu$ is any element of $\mu$ and $\xi_\mu$ is a one dimensional representation of $Z(c_\mu)$.
		Thus we have:
		$$\dim(H^i(\pconf_n(\C);\C)) = \sum_{\mu \in T(i,n)} \dim \Ind_{Z(c_\mu)}^{S_n} (\xi_\mu) = \sum_{\mu \in T(i,n)} [S_n : Z(c_\mu)].$$
		Write $\mu$ in cycle notation $1^{\mu_1} 2^{\mu_2} \dots r^{\mu_r}$. Recall that the centralizer of $c_\mu$ in $S_n$ can be expressed as:
		$$Z(c_\mu) = S_{\mu_1} \rtimes \bigg((\Z/2)^{\mu_2} \times S_{\mu_2}\bigg) \dots \bigg((\Z/r)^{\mu_r} \rtimes S_{\mu_r}\bigg).$$
		And in particular,
		$$|Z(c_\mu)| = \prod_j \mu_j! j^{\mu_j}.$$
		Therefore,
		\begin{equation} \label{dim_formula}
			\dim(H^i(\pconf_n(\C);\C)) =\sum_{\mu \in T(i,n)} [S_n : Z(c_\mu)] = \sum_{\mu \in T(i,n)} \frac{n!}{\prod_j \mu_j! j^{\mu_j}} .
		\end{equation}
		Now assume that $n \geq 2i$, so $\phi:T(i,n) \hookrightarrow P(i)$ is a bijection. Then after identifying $T(i,n)$ and $P(i)$ via $\phi$, we have
		$$\dim(H^i(\pconf_n(\C);\C)) = \sum_{\la \in P(i)} \frac{1}{\prod_{j=1}^n \la_j! (j+1)^{\la_j}} \frac{n!}{(n - i - \sum_{j} \la_j)!}.$$
		The advantage of this identification is that now the sum is over $P(i)$ and thus independent of $n$. 
		Now fix some partition $\la = 1^{\la_1} \dots r^{\la_r}$ of $i$. The expression
		\begin{equation}\label{dim_expression}
			\frac{1}{\prod_{j=1}^n \la_j! (j+1)^{\la_j}} \frac{n!}{(n - i - \sum_{j} \la_j)!}
		\end{equation}
		is a polynomial in $n$ (for $n \geq 2i \geq i + \sum_j \la_j $) of degree $i + \sum_j \la_j$ with leading coefficient $\frac{1}{\prod_{j=1}^n \la_j! (j+1)^{\la_j}}$. Therefore, $\dim (H^i(\pconf_n(\C);\C))$ is given by a polynomial $f_i(n)$ in $n$ for all $n \geq 2i$. Furthermore, the degree of $f_i(n)$ is equal to the maximum value of $i + \sum_j \la_j$ among $\la \in P(i)$. This value is maximized when $\la = 1^{i}$, in which case the degree of Expression \ref{dim_expression} as a polynomial in $n$ is exactly $2i$. Therefore for $n \geq 2i$, $\dim (H^i(\pconf_n(\C);\C))$ is given by a polynomial in $n$ of degree $2i$.
	\end{proof}
	
	\begin{lemma} \label{dimV(la)_n}
		Let $\la = (a_1, \dots, a_r)$ be a Young diagram with $k$ boxes. Then $\dim V(\la)_n$ is a degree $k$ polynomial $g_\la(n)$ in $n$. Furthermore, the leading coefficient of $g_\la(n)$ is $\frac{\dim \la}{k!}$, where $\dim \la$ is the dimension of the irreducible representation given by $\la$.
	\end{lemma}
	\begin{proof}
		We compute the character polynomial $\chi^\la$ of $V(\la)$ by Equation \ref{youngtocharpoly} and use the fact that $\dim V(\la)_n = \chi^\la(\id_{S_n})$. We rewrite Equation \ref{youngtocharpoly} for convenience:
		$$
		\chi^\la = \sum_{|\rho| \leq |\la|} \binom{X}{\rho}  (-1)^{|\la| - |\rho|} \sum_\mu \chi_\rho^\mu.
		$$
		Recall that $\mu$ is indexed over all partitions $|\mu| = |\rho|$ such that $\la - \mu$ is a \emph{vertical strip}, that is, $\la$ can be obtained by adding $|\la| - |\mu|$ boxes to $\mu$ without adding more than one box in any one row.
		\vspace{.1in}\par
		The trivial permutation $\Id_{S_n}$ has cycle type $1^n 2^0 \dots n^0$, and thus when applying $\Id_{S_n}$ to $\chi^\la$, all the binomial terms $\binom{X}{\rho}(\Id_{S_n})$ vanish except for those with $\rho = (1, 1, \dots, 1) = 1^b$, in which case $\binom{X}{\rho}(\Id_{S_n}) = \binom{n}{b}$. Therefore,
		$$\chi^\la(\Id_{S_n}) = \sum_{b=0}^{|\la|} (-1)^b \binom{n}{b} \sum_{\mu} \chi^\mu_{[1^b]}.$$
		Thus, $\chi^\la(\Id_{S_n})$ is a polynomial in $n$ of degree at most $k = |\la|$. To show its degree is exactly $k$ and to compute the coefficient of $n^{k}$, it suffices to consider the coefficient $\sum_\mu \chi^\mu_{[1^{k}]}$ of $\binom{n}{k}$. The only $\mu$ with $|\mu| = |\la|$ such that $\la - \mu$ is a vertical strip is $\mu = \la$, so the coefficient of $\binom{n}{k}$ is $\chi^\la_{[1^k]} = \dim \la$. Thus, $g_\la(n)$ is a polynomial of degree $k$ with leading coefficient $\frac{\dim \la}{k!}$ as desired.
	\end{proof}
	
	\subsection{Asymptotic Behavior of Multiplicities}
	
	In this subsection we prove Theorem \ref{asymptotic}.
	Our proof follows from approximating the coefficient of $z^m$ in each formal power series $\Phi^\infty_\la(z)$ of Equation \ref{polystat}. First we introduce notation for the asympototic behavior of coefficients of formal Laurent series.
	
	\begin{definition}
		Let $f(z) = a_{-m}z^{-m} + \dots + a_0 + a_1 z + a_2 z^2 + \dots$ be a formal Laurent series in $z$ with coefficients $\{a_i\}_{i=-m}^\infty$ in $\C$. Let $r \geq 0$ be an integer. We say that $f$ is \textbf{$r$-bounded} if there exists $C \in \R^+$ such that for all but finitely many $n \in \Z$, $|a_n| \leq Cn^r$. In other words, the function $n \mapsto a_n$ is $O(n^r)$. 
	\end{definition}
	\renewcommand{\labelenumi}{(\alph{enumi})}
	\begin{lemma} \label{boundedarithmetic}
		Let $f, g$ be formal Laurent series in $z$ with coefficients in $\C$ and assume $f$ is $r$-bounded and $g$ is $s$-bounded. Let $h \in \C[z, z^{-1}]$ be a Laurent polynomial.
		\begin{enumerate}
			\item $f + c\cdot g$ is $\Big(\max{(r + s)}\Big)$-bounded for any $c \in \C$. 
			\item $f \cdot h$ is $r$-bounded.
			\item $f \cdot g$ is $(r + s + 1)$-bounded.
		\end{enumerate}
	\end{lemma}
	
	\begin{proof}
		We use the notation $[z^n]f$ for the coefficient of $z^n$ in $f$.
		\begin{enumerate}
			\item[(a)] Trivial.
			\item[(b)] Let 
			$$f(z) = a_{-m}z^{-m}+ \dots + a_0 + a_1z + \dots,$$ 
			and let
			$$h(z) = c_{-N}z^{-N} + \dots + c_Nz^N.$$ Since $f$ is $r$-bounded, let $C \in \R^+$ such that $|a_n| \geq Cn^r$ for all but finitely many $n$. Since $f \cdot h = \sum_{i=-N}^N c_i z^i f$, by part (a) it suffices to consider the case of $h(z) = z^N$ for $N \in \Z$. Notice that the function
			$$\psi:\Z\setminus \{0\} \to \Q \hspace{.5in} \psi(n) = \frac{(n - N)^r}{n^r}$$
			has bounded image in $\Q$ since $\lim_{n \to \infty} \psi(n) = \lim_{n \to -\infty} \psi(n) = 1$. Thus, let $D \in \R^+$ such that 
			$$C \frac{(n-N)^r}{n^r} \leq D$$
			for all $n \in \Z \setminus \{0\}$. Then for all but finitely many $n$,
			$$\Big|[z^n] (z^N \cdot f )\Big| = |a_{n- N}| \leq C(n - N)^r \leq D n^r.$$
			Thus, $z^N \cdot f$ is $r$-bounded.
			\item[(c)] Let 
			$$f(z) = a_{-m}z^{-m}+ \dots + a_0 + a_1z + \dots \hspace{.2in} g(z) = b_{-m}z^{-m}+ \dots + b_0 + b_1z + \dots$$ 
			be two formal Laurent series. Without loss of generality, we assume $m = -1$ for ease of notation, since if $m \neq -1$ then using (b) we may multiply by $z^{m-1}$ and then $z^{-m+1}$. Let $C, D \in \R^+$ such that $|a_n| \leq Cn^r$ and $|b_n| \leq Dn^s$ for all $n \geq 1$. Then for all $n \geq 1$ we have
			$$[z^n](f \cdot g) = \sum_{l = 1}^{n} a_{l}b_{n - l} \leq CD \sum_{l=1}^{n} |l|^r |n-l|^s \leq CD \sum_{l=1}^{n} n^r n^s = CD n^{r + s + 1} $$
			Therefore,
			$$[z^n](f \cdot g) \leq CD n^{r+s+1}$$
			for all $n$, so $f \cdot g$ is $(r + s + 1)$-bounded.
		\end{enumerate}
	\end{proof}
	
	\begin{lemma} \label{bound_phi}
		Let $\rho = 1^{j_1} \dots r^{j_r}$ be a partition of $k$ with $j_r \neq 0$. Then $\Phi^\infty_\la(z)$ is $(k - r)$-bounded. 
	\end{lemma}
	
	\begin{proof}
		We repeat Equation \ref{polystat} defining $\Phi_\rho^{\infty}(z)$ for convenience:
		$$\Phi_\rho^{\infty}(z) = (1 - z)\prod_{i=1}^r \binom{M_{i}(z^{-1})}{j_i} (z^{i} - z^{2i} + \dots )^{j_i}.$$
		Note that for $l \in \N$, the formal Laurent series $(z^l - z^{2l} + \dots)$ is $0$-bounded (with $C = 1$). Therefore by Lemma \ref{boundedarithmetic},
		$(z^l - z^{2l} + \dots)^j$ is $(j-1)$-bounded. $\Phi_\rho^\infty$ is thus a product of a Laurent polynomial with a $j_1-1$-bounded, $\dots$, $j_r-1$-bounded Laurent series and thus is $\Big(-1 + \sum_{i=1}^r j_i\Big)$-bounded by Lemma \ref{boundedarithmetic}. Since $\sum_{i=1}^r ij_i = k$ and $j_r \geq 1$, we have that 
		$$r-1 + \sum_{i=1}^rj_i \leq (r-1)j_r + \sum_{i=1}^r j_i \leq \sum_{i=1}^r ij_i = k$$
		Therefore, $\sum_{i=1}^r j_i \leq k - r + 1$, so $\Phi_\rho^{\infty}(z)$ is $(k-r)$-bounded.
	\end{proof}

	\begin{lemma}\label{k-asym}
		Fix a positive integer $k$. Let $a_n$ be the $n$th Taylor coefficient of $\Phi^\la_{1^k}(-z)$ so $\Phi^\la_{1^k}(z) = a_0 - a_1 z + a_2 z^2 - \dots$ as a formal power series. Then there exists a polynomial $f(n)$ of degree $k-1$ with leading coefficient $\frac{2}{k!}$ such that $a_n = f(n)$ for $n \geq k$.
	\end{lemma}
	
	\begin{proof}
		Let $w = -z$ for ease of notation. We have:
		$$\Phi^\infty_{1^k}(z) = \Phi^\infty_{1^k}(-w) = (1+w)\binom{M_1(-w^{-1})}{k}(-w - w^2 + \dots)^k $$
		
		$$ = (1+w)\frac{w^{-1}(w^{-1} + 1) \dots (w^{-1} +k -1)}{k!}(w + w^2 + \dots)^k.$$
		By Equation \ref{coefequation},
		$$(w + w^2 + \dots)^k = \sum_{j=k}^\infty \binom{j - 1}{j - k} w^j.$$
		Let $c(n,k) = |s(n, k)|$ denote the $n,k$th unsigned Sterling number, so that
		$$w^{-1}(w^{-1} + 1) \dots (w^{-1} + k -1) = \sum_{a = 0}^k w^{-a}c(k,a).$$
		Then we have that
		$$a_n = [w]^n \Phi^\infty_{1^k}(-w) = \sum_{a = 0}^k \frac{c(a, k)}{k!} \binom{n + a - 2}{n + a - k - 1} +\sum_{a = 0}^k \frac{c(a, k)}{k!} \binom{n + a - 1}{n + a - k}.$$
		As functions of $n$, each of $\binom{n + a -2}{n + a - k -1}, \binom{n + a - 1}{n + a - k}$ are polynomials in $n$ of degree $k - 1$ with leading coefficient $\frac{1}{(k-1)!}$, as long as $(n + a - k - 1) \geq 0$, i.e., $n > k$. Therefore, the coefficients $a_n$ are given by a polynomial $f(n)$ for $n > k$.
		To complete the proof of the Lemma we compute the coefficient of $n^{k-1}$ in $f(n)$. Since $f(n)$ is expressed as a sum of polynomials of degree $k-1$, $\deg f = k -1$ so long as the sum of their leading coefficients is non-zero. Recall that $\sum_{a=0}^k c(a,k) = 1(1 + 1)(1 + 2) \dots (1 + k -1) = k!$. Therefore,
		$$[n^{k-1}]f(n) = \sum_{a=0}^k \frac{c(a,k)}{k!} \frac{1}{(k-1)!} + \sum_{a = 0}^k \frac{c(a,k)}{k!} \frac{1}{(k-1)!} = \frac{2}{(k-1)!},$$
		so $f$ is degree $k-1$ and has leading coefficient $\frac{2}{(k-1)!}$, proving the Lemma.
	\end{proof}
	Recall Equation \ref{chen_formal_equation} which gives $ (-1)^i d_i(\la) z^i$ as a rational combination of the formal power series $\Phi^\infty_\rho(z)$ for $|\rho| \leq |\la|$:
	$$\sum_{i \geq 0}(-1)^i d_i(\la) z^i = \sum_{|\rho| \leq |\la|} F^\la_\rho \Phi_\rho^{\infty}(z).$$
	By Lemma \ref{bound_phi}, all of the formal power series $\Phi^\infty_{\rho}(z)$ are $(k - 2)$-bounded except for possibly $\Phi^\infty_{1^k}(z)$. By Lemma \ref{k-asym}, the coefficients $a_n$ of $\Phi^\infty_{1^k}$ are asymptotically $\frac{2 n^{k-1}}{(k-1)!}$, which thus dominates the asymptotic behavior of $\lim_{i\to\infty} d_i(\la)$ as long as $F^\la_{1^k}$ is non-zero. We compute $F^\la_{1^k} =  \chi^\la_{1^k}$ using Equation \ref{youngtocharpoly}.
	Recall that $\chi^\la_{1^k}$ is the dimension of the irreducible representation given by $\la$, and thus is always non-zero so long as $\la$ is nonempty. Therefore, 
	$$\lim_{i \to \infty} \frac{d_i(\la)}{i^{k-1}} = \lim_{i \to \infty} F^\la_{1^k}\frac{2i^{k-1}}{(k-1)! i^{k-1}} + \frac{o(i^{k-1})}{i^{k-1}} = \frac{2\dim \la}{(k-1)!} \neq 0,$$
	which completes the proof of Theorem \ref{asymptotic}.
	
	\subsection{Bound of Length of Tableaux}
	
	In this subsection we prove Theorem \ref{length_bound}. We would like to thank an anonymous reviewer again for this proof. 
	
	\begin{definition}
		Let $V$ be a representation of $S_n$. Let $\ell(V)$ be the longest length of all partitions
		corresponding to irreducible summands of $V$, called the \textit{length} of $V$ as an $S_n$ representation.
	\end{definition}
	To prove Theorem \ref{length_bound}, it suffices to show the following unstable statement.
	
	\begin{proposition}
		$\ell(H^i(\pconf_n(\C);\C)) \leq i + 1$ for all $i$ and $n$. 
	\end{proposition} 
	
	\begin{proof}
		We prove the statement using Lehrer and Solomon's description of $H^i(\pconf_n(\C);\C)$ as an $S_n$ representation \cite[Equation 1.3]{lehrersolomon}:
		$$H^i(\pconf_n(\C);\C) \cong \bigoplus_{\substack{|\la| = n\\\ell(\la)=n-i}} \Ind_{S_\la}^{S_n} L(\la),$$
		where, if $\la = 1^{r_1}2^{r_2} \dots$ is the cycle notation of $\la$, then $S_\la := \prod_{k\geq 1}(S_k^{r_k} \rtimes S_{r_k})$ with the semidirect product induced by permutation $S_{r_k} \circlearrowright (S_k^{r_k})$, i.e., the wreath product $S_k^{r_k} \wr S_{r_k}$. 
		$L(\la)$ refers to the representation of $S_\la$ given by	$\otimes_{k\geq1}L(k)^{r_k}$, where $L(k) = \sgn_k \otimes \Ind_{C_k}^{S_k}\ze$, where $C_k \subset S_k$ is the cyclic subgroup generated by a $k$-cycle and $\ze$ denotes the character of $C_k$ given by multiplication by a primitive $k$th root of unity. Equivalently, $L(k) = \sgn_k \otimes\  \mathsf{Lie}(k)$ where $\mathsf{Lie}$ is the Lie operad.
		\vspace{.1in}\par
		Now let $S'_\la$ be the subgroup of $S_\la$ given by $S_{r_1} \times \prod_{k\geq 2} S_k^{r_k}$. Unlike $S_\la$, the subgroup $S'_\la$
		is a Young
		subgroup of $S_n$. Since
		$$\Ind_{S_\la}^{S_n} L(\la) \subseteq 
		\Ind^{S_n}_{S'_\la} \Res^{S_\la}_{S_\la'} L(\la),$$
		it suffices to show that $\ell(\Ind^{S_n}_{S'_\la} \Res^{S_\la}_{S_\la'} L(\la)) \leq i+1$. We use the following facts:
		\renewcommand{\labelenumi}{(\arabic{enumi})}
		\begin{enumerate}
			\item \label{fact_1} If $V$ is a representation of $S_l$ and $W$ is a representation of $S_k$, then 
			$$\ell(V) + \ell(W) = \ell(\Ind_{S_k \times S_l}^{S_{k+l}} V \otimes W).$$
			
			\item \label{fact_2} If $k \geq 2$ then $\ell(L(k)) = k-1$, and $\ell(L(1)) = 1$.
		\end{enumerate}
		
		These facts together, along with the fact that $\Res^{S_\la}_{S'_\la}L(k)^{r_k} \cong \bigotimes_{i=1}^{r_k} L(k)$ as an $S_k^{r_k}$ representation imply that 
		$$\ell(\Ind_{S'_\la}^{S_n}\Res_{S'_\la}^{S_\la} L(\la)) = r_1 + \sum_{k \geq 2} (k-1)r_k.$$
		Recalling that $\sum_k kr_k = n$ and $\sum_k r_k = n-i$, we have
		$$\ell(\Ind_{S'_\la}^{S_n}\Res_{S'_\la}^{S_\la} L(\la)) = \begin{cases}
			i & r_1 = 0\\
			i+1 & r_1 > 0
		\end{cases}$$
		\vspace{.1in}\par
		Thus, it only remains to prove Facts \ref{fact_1} and \ref{fact_2}.
		Fact \ref{fact_1} follows from the Littlewood-Richardson rule. To prove Fact \ref{fact_2}, we use Frobenius reciprocity twice. The case of $k = 1$ is trivial, so let $k \geq 2$ be a positive integer.
		\vspace{.1in}\par
		To show that $\ell(L(k)) < k$, it suffices to show that $L(k)$ does not contain a summand isomorphic to the sign representation, as $\sgn_{k}$ is the only length $k$ irreducible representation of $S_k$. We have
		$$\langle L(k), \sgn_k \rangle_{S_k} = \langle \Ind_{C_k}^{S_k} \ze, \Id_{S_k} \rangle_{S_k} = \langle \ze, \Id_{C_k} \rangle = 0,$$
		where $\Id_{S_k}$ is the trivial character on $S_k$ and $\Id_{C_k}$ is the trivial character on $C_k$. Thus, $\sgn_k$ is not a subrepresentation of $L(k)$.
		\vspace{.1in}\par
		To show that $\ell(L(k)) \geq k-1$, it suffices to show that $L(k)$ contains a summand isomorphic to $\sgn_k \otimes V$, where $V$ is the standard representation of $S_k$. Then using that $\Res^{S_k}_{C_k} V \cong \bigoplus_{i=1}^{k-1} \ze^{\otimes i}$, we have
		$$\langle L(k), \sgn_k \otimes V \rangle_{S_k} \cong \langle \Ind_{C_k}^{S_k} \ze, V \rangle_{S_k} \cong \Big\langle \ze, \bigoplus_{i=1}^{k-1} \ze^{\otimes i} \Big\rangle_{C_k} = 1.$$
		\par In conclusion, $\ell(L(k)) = k-1$, as desired.
	\end{proof}

	\appendix 
	
	\section{Character Polynomials} \label{charpolysappendix}
	
	One reason why the definition of the irreducible families $V(\la)$ is natural is because the character of each irreducible $V(\la)_n$ is given simultaneously by a single \textit{character polynomial}. In fact, the phenomenon of representation stability in $H^i(\pconf_n(\C);\C)$ is equivalent to its character as an $S_n$ representation eventually being given by a single character polynomial as $n \to \infty$. 
	
	\begin{definition}
		Define functions $\Big\{X_i:\bigsqcup_{n \geq 1}{S_n} \to \Z\Big\}_{i \geq 1}$ by the following:
		for an element $\si \in S_n$, let
		$$X_i(\si) := \text{(number of $i$-cycles in $\si$).}$$
	\end{definition}
	
	\begin{definition} \textbf{(Character Polynomial).}
		Let $P \in \Q[X_1, X_2, \dots, ]$ be a polynomial. Define $P:\bigsqcup_{n \in \N}S_n \to \Q $ by
		$$P(\si) := P(X_1(\si), X_2(\si), \dots, X_n(\si), 0, 0, \dots).$$
		for $\si \in S_n$.
		With respect to the function $P:\bigsqcup_n S_n \to \Q$, $P$ is called a \textit{character polynomial}. The degree of a character polynomial is defined by $\deg X_k = k$ and extended to arbitrary polynomials by linearity.
	\end{definition}
	
	\begin{definition}\textbf{(Binomial Basis)} \label{binombasis}
		Let $\rho = 1^{\rho_1} 2^{\rho_2} \dots r^{\rho_r}$ be a partition of $n$. This notation means that $\rho$ contains $\rho_i$ parts of size $i$, so $\sum_i i \rho_i = n$. Define a character polynomial $\binom{X}{\rho}$ by
		$$\binom{X}{\rho} = \binom{X_1}{\rho_1} \binom{X_2}{\rho_2} \dots \binom{X_r}{\rho_r}.$$
		The set
		$$\bigg\{ \binom{X}{\rho}\bigg\}$$
		indexed over all partitions $\rho$ is called the \textit{binomial basis} of $\Q[X_1, X_2, \dots]$.
	\end{definition}
	The binomial basis is convenient for computations because it naturally arises for character polynomials of irreducible representations and for polynomial statistic formulas. 
	
	For any Young diagram $\la$ with $k$ boxes, there is a unique character polynomial $P$ of degree $k$ which is simultaneously the character of every irreducible $V(\la)_n$. For instance when $\la = \yng(1)$, $V(\la)$ is the family of standard representations of $S_n$, and $\chi^\la = \binom{X_1}{1} - 1$. This fact is a consequence of the Frobenius character formula \cite[Formula 4.10]{fulton_and_harris} and was first observed by Frobenius in the early 1900s \cite[pg. 134]{macdonald1998symmetric}. We originally implemented an algorithm based on the Frobenius character polynomial $\chi^\la$. However, this algorithm is almost always less efficient than the formula given by Macdonald in Theorem \ref{charpolyformula}. The only cases where they are comparable is when $\la$ is of length $2$ or less.
	\vspace{.1in}\par
	The notation $\chi^\mu_\rho$ in Theorem \ref{charpolyformula} refers to evaluating the irreducible representation of the partition $\mu$ (by the Young bijection) on the conjugacy class of $S_n$ given by $\rho$. In order to efficiently compute the coefficients $\chi^\mu_\rho$ we utilize the recursive Murnaghan-Nakayama rule. 
	
	\begin{theorem} \textbf{(Murnaghan-Nakayama Rule \cite[ex 1.7.5]{macdonald1998symmetric})}
		Let $\mu = (\mu_1, \dots, \mu_r)$,  $\rho = (\rho_1, \dots, \rho_s)$ be partitions of $n$. Then
		\begin{equation}\label{murnaghan_nakayama}
			\chi^\mu_\rho = \sum_{\xi \in BS(\mu,\rho_1)} (-1)^{ht(\xi)}\chi^{\mu - \xi}_{\rho \setminus \rho_1},
		\end{equation}
		where $BS(\mu, \rho_1)$ is the set of border strips $\xi$ within $\mu$ of exactly $\rho_1$ boxes such that $\mu - \xi$ is still a valid Young diagram. $ht(\xi)$ is the number of rows $\xi$ touches in $\mu$ minus $1$.
	\end{theorem}

	\section{Implementation} \label{alg-details}
	Theorems \ref{charpolyformula} and \ref{polystatformula} give a procedure to compute the coefficients $d_i(\la)$. 
	The complete algorithm may now be described succinctly by StableCoefficients, with input $\la = (\la_1, \dots, \la_r)$ a tuple of nonincreasing positive integers representing a Young Tableau and $maxDegree$ a positive integer representing what exponent the power series in $z$ will be computed to.
	\begin{algorithm} 
		\caption{StableCoefficients ($\la$, $maxDegree$)}
		\begin{algorithmic}
			\State $\chi^\la \gets \text{YoungToCharPoly}(\la)$ \Comment{Algorithm (\ref{youngtocharpolyalg})}
			
			\State \Return $\text{PolynomialStatistic}(\chi^\la, maxDegree)$ \Comment{Algorithm (\ref{polystatalg})}
		\end{algorithmic}
	\end{algorithm}
	\begin{algorithm} 
		\caption{YoungToCharPoly ($\la$)} \label{youngtocharpolyalg}
		\begin{algorithmic}
			\State AllMu$ \gets \{\mu \subset \la \ \Big| \ \la - \mu \text{ is a vertical strip}\}$
			\State $P^\la \gets 0$
			\ForAll{Partitions $\rho$ with $|\rho| \leq |\la|$}
			\State $F^\la_\rho \gets 0$
			\ForAll{$\mu \in \text{AllMu}$ with $|\mu| = |\rho|$}
			\State $F^\la_\rho \gets F^\la_\rho + (-1)^{|\la| - |\rho|}\text{CharEval}(\mu, \rho)$ \Comment{Algorithm \ref{chareval}}
			\EndFor
			\State $P^\la \gets P^\la + F^\la_\rho \binom{X}{\rho}$
			\EndFor
			\State \Return $P^\la$
		\end{algorithmic}
	\end{algorithm}
	
	When initializing AllMu in Algorithm \ref{youngtocharpolyalg}, we iterate over all partitions $\mu$ with $|\mu| \leq |\la|$, and check whether $\mu \subset \la$ and $\la - \mu$ is a vertical strip. There are more efficient solutions, but since we already iterate over all partitions $\rho$ with $|\rho| \leq |\la|$ in the following loop, this does not affect the asymptotic complexity of the algorithm.
	
	\begin{algorithm} 
		\caption{CharEval ($\mu, \rho$)} \label{chareval}
		\begin{algorithmic}
			\If{$|\mu| \leq 1$}
			\State \Return $1$
			\EndIf
			
			\State $C \gets 0$
			\ForAll{$\xi \in BS(\mu, \rho_1)$}
			\State $C \gets C + (-1)^{ht(\xi)}\text{CharEval}(\mu - \xi, \rho - \rho_1)$
			\EndFor
			\State \Return $C$
		\end{algorithmic}
	\end{algorithm}
	
	To iterate over $BS(\mu, \rho_1)$ in Algorithm \ref{chareval}, notice that any $\xi \in BS(\mu, \rho_1)$ must consist of the rightmost box of one row of $\mu$ and then zig-zag downwards along the rightmost side of $\mu$. Thus, any $\xi \in BS(\mu, \rho_1)$ is determined by the topmost row it touches. Thus, we test each possible starting row to find all $\xi$.
	
	\begin{algorithm} 
		\caption{PolynomialStatistic ($\chi^\la, maxDegree$)}\label{polystatalg}
		\begin{algorithmic}
			\State $\text{result} \gets 0$
			\ForAll{monomial terms $a \binom{X}{\rho}$ of $\chi^\la$ with $\rho = 1^{\rho_1} \dots s^{\rho_s}$}
			\State $R \gets a$
			\For{$1 \leq t \leq s$}
			\State $R \gets R \cdot \binom{M_{t}(z^{-1})}{\rho_t}(z^{t} - z^{2t} + \dots )^{\rho_t}$
			\Comment{(computed up to $q^{maxDegree}$)}
			\EndFor 
			\State $\text{result} \gets \text{result} + (1-z)R$
			\EndFor
			\State \Return result
		\end{algorithmic}
	\end{algorithm}
	Notice that $\chi^\la$ is naturally computed in Algorithm \ref{youngtocharpolyalg} in terms of the binomial basis \ref{binombasis}. Thus, we store $\chi^\la$ in terms of the binomial basis as input for PolynomialStatistic.
	
	\subsection{Correctness and Efficiency}
	
	\renewcommand{\labelenumi}{\textbf{\arabic{enumi}.}}
	\begin{enumerate}
		\vspace{.1in}
		\item \textit{Computing the output of Algorithm \ref{polystatalg} accurately to maxDegree terms.}
		\vspace{.1in}\par
		We rely on two basic facts of formal power series arithmetic. Suppose that 
		$$f(z) = a_{m}z^{-m} + a_{m-1}z^{-m+1} + \dots  \hspace{.5in} g(z) = b_nz^{-n} + b_{n-1}z^{-n+1} + \dots$$
		are two formal power series in $z$ with some negative exponent terms. 
		\vspace{.1in}\par
		\textbf{Fact 1. } To compute $f \cdot g$ up to the coefficient of $z^{maxDegree}$, it suffices to compute $f_{maxDegree + n}(z) \cdot g_{maxDegree + m}(z)$, where $f_{k}$ is the finite Laurent series formed by truncating $f$ to degree $k$ in $z$. 
		\vspace{.1in}\par
		\textbf{Fact 2.} To compute $f + g$ up to the coefficient of $z^{maxDegree}$, it suffices to compute $f_{maxDegree} + g_{maxDegree}$. 
		\vspace{.1in}\par
		Since $\binom{M_t(z^{-1})}{\rho_t}$ has degree $t \rho_t$ in $z^{-1}$, and $(z^t - z^{2t} + \dots)^{\rho_t}$ is a power series with leading degree $t \rho_t$, each of the terms $\binom{M_t(z^{-1})}{\rho_t} (z^t - z^{2t} + \dots)^{\rho_t}$ is a power series with all nonnegative exponents of $z$. Thus, it suffices (by Facts 1 and 2) to compute each power series $\binom{M_t(z^{-1})}{\rho_t} (z^t - z^{2t} + \dots)^{\rho_t}$ up to the coefficient of $z^{maxDegree}$. By Fact 1, this product can be computed by truncating $(z^t - z^{2t} + \dots)^{\rho_t}$ to degree ${maxDegree + t}$. The coefficient of $z^{lt}$ in $(z^t - z^{2t} + \dots)^{\rho_t}$ is given by 
		\begin{equation} \label{coefequation}
			[z^{lt}]\Big(z^t - z^{2t} + \dots\Big)^{\rho_t} = (-1)^{l - \rho_t}\binom{l-1}{l - \rho_t}.
		\end{equation}

		\vspace{.1in}
		\item \textit{Memoization.} 
		\vspace{.1in}\par
		The recursive Murnaghan-Nakayama rule \ref{murnaghan_nakayama} relies on recursively computing $\chi^\la_\rho$ and thus lends itself to memoization. Each time a character $\chi^\la_\rho$ is computed, we save the result for later calls to the same function. Otherwise we apply the recursive Murnaghan-Nakayama rule, leading to further function calls of Algorithm \ref{chareval}. In practice we build the character tables of $S_m$ in order from $m = 0, 1, 2,\dots$ so the recursive depth never reaches more than $1$. 
		\vspace{.1in}\par
		In Algorithm \ref{polystatalg}, we repeatedly compute power series 
		$$\psi_t^{\rho_t} := \binom{M_t(z^{-1})}{\rho_t} (z^t - z^{2t} + \dots)^{\rho_t} \hspace{.5in} \Phi_\rho^\infty = (1-z) \sum_{t=1}^r \psi_t^{\rho_t}.$$
		Thus, we store $\psi_t^{\rho_t}$ and $\Phi_\rho^\infty$ each time they are computed. Furthermore when computing $\psi^{\rho_t}_t$, we repeatedly compute the coefficients of $ (z^t - z^{2t} + \dots)^{\rho_t}$.
		These coefficients are the same independent of $t$, so we store the result (of Equation \ref{coefequation}) to avoid recomputing them. Finally, to compute the power series $\binom{M_t(z^{-1})}{\rho_t}$, we use the basic recursion
		$$\binom{M_t(z^{-1})}{\rho_t} = \frac{M_t(z^{-1}) - \rho_t + 1}{\rho_t} \binom{M_t(z^{-1})}{\rho_t-1}$$
		to compute $\binom{M_t(z^{-1})}{\rho_t}$ from $\binom{M_{t-1}(z^{-1})}{\rho_t-1}$, storing previous values for speedup.
		\vspace{.1in}
		\item \textit{Correctness.} 
		\vspace{.1in}\par
		The correctness of the algorithm is due to the proofs within this paper along with the work of Chen and Macdonald. Additionally, the results agree with previous examples computed as discussed in Section \ref{past_data}.
		In order to eliminate arithmetic errors like overflow or floating point, we utilize the \href{https://learn.microsoft.com/en-us/dotnet/api/system.numerics.biginteger?view=net-8.0}{BigInteger} class in C\# to implement a BigRational class, which stores numerator and denominator as BigIntegers. Additionally, we have written an extensive suite of test cases built from by hand examples for every step of the algorithm. The code is available along with test cases on \href{https://github.com/LimeHero/RepresentationStability}{GitHub}. Additionally, the outputs satisfy expected combinatorial conditions, like the coefficient of $z^i$ being an integer and of sign $(-1)^i$, which are usually not satisfied with any minor errors in the code.

		\vspace{.1in}
		\item \textit{Algorithmic Complexity.} \vspace{.1in}\par
		Let $p(n)$ be the number of integer partitions of $n$ and let $A(n) = \sum_{i=0}^n p(n)$.
		For simplicity, we assume that arithmetic between elements of the BigRational class is $O(1)$. Note that elements of the BigRational class are \textit{always} reduced fractions (numerator and denominator are relatively prime). In particular, after any arithmetic operation between rational numbers, we reduce the result utilizing the \href{https://learn.microsoft.com/en-us/dotnet/api/system.numerics.biginteger.greatestcommondivisor?view=net-8.0}{GreatestCommonDivisor} method. We note that there may be more efficient implementations. \vspace{.1in}\par
		We consider the algorithmic complexity of computing StableCoefficients($\la, maxDegree$) for all $|\la| \leq n$. Through memoization, the coefficients $\chi^\mu_\rho = \text{CharEval}(\mu, \rho)$ are only computed once for each $|\mu| = |\rho|$, and the power series $\Phi_\rho^\infty(z)$ are only computed once each for each $|\rho| \leq n$. Thus, we consider the complexity of each of these steps separately. In particular, it suffices to take the sum of the algorithmic complexity of: \textbf{(i)} solving for each power series $\Phi_\rho^\infty(z)$ to $maxDegree$ terms, \textbf{(ii)} determining the character tables of $S_m$ for $m \leq n$, \textbf{(iii)} determining the character polynomials $\chi^\la$ for all $|\la| \leq n$ given the result of (ii), \textbf{(iv)} Algorithm \ref{polystatalg} for all $|\la| \leq n$ given the results of (i) and (iii).
		\renewcommand{\labelenumii}{\textbf{(\roman{enumii})}}
		\begin{enumerate}
			\item We consider the algorithmic complexity of determining each $\Phi_\rho^\infty(z)$ for $|\rho| \leq n$. We note that we only need to compute the coefficients of $(z - z^{2} + \dots)^{k}$ up to $maxDegree + n$ and for all $k \leq n$ once. By Equation \ref{coefequation} given for the coefficients of $(z - z^{2} + \dots)^{k}$ computing any one such coefficient is bounded by $O(n)$. For each power series we compute $maxDegree + n$ terms and there are $n$ total power series, so altogether this takes $O(n(maxDegree + n))$. This is dominated by other terms so we may ignore it. Similarly by the memoization of the terms $\binom{M_t(z^{-1})}{\rho_t}$ we only compute each such binomial term once. By definition,
			$$\binom{M_t(z^{-1})}{\rho_t} = \frac{M_t(z^{-1})(M_t(z^{-1}) - 1) \dots (M_t(z^{-1}) - \rho_1 + 1)}{\rho_t!}.$$
			The numerator involves $\rho_t$ products of polynomials of degree at most $t$. Since $t \cdot \rho_t \leq |\rho| \leq n$, the complexity of computing this expression for a pair $(t, \rho_t)$ is $O(n^2)$, and computing this polynomial for all such $(t, \rho_t)$ is thus $O(n^3)$. We repeat Equation \ref{polystat} again for convenience:
			$$\Phi^\infty_\rho(z) = (1-z)\prod_{t=1}^r \binom{M_t(z^{-1})}{\rho_t}(z^t - z^{2t} + \dots)^{\rho_t}.$$
			Given these results, we have that $\Phi^\infty_\rho(z)$ is a product of at most $|\rho|$ formal polynomials each with at most $|\rho|$ and $maxDegree$ terms respectively. Thus for any such $\rho$, computing $\Phi^\infty_\rho$ is $O(|\rho|^2 maxDegree)$, and so computing all such polynomials is 
			$$O(A(n)n^2maxDegree).$$
			
			\item The algorithmic complexity of finding the character table of $S_m$ for all $m \leq n$ via the Murnhaghan-Nakayama rule is dominated by finding all border strips $BS(\mu, \rho_1)$ for all pairs $\mu, \rho$. Given any pair $\mu, \rho$, the algorithmic complexity of finding $BS(\mu, \rho_1)$ is $O(r)$ where $r$ is the lenght of $\mu$. Furthermore, there are at most $A(n)^2$ pairs $\mu, \rho$. Thus, the algorithmic complexity of finding all such character tables is
			$$O(A(n)^2n).$$
			
			\item Now we consider the complexity of Algorithm \ref{youngtocharpolyalg} run on all $\la$ with $|\la| \leq n$, given the result of (ii).We iterate over all $\rho$ with $|\rho| \leq |\la|$, and then all $\mu$ with $|\mu| = |\rho|$ such that $\la - \mu$ is a vertical strip. Thus, the algorithmic complexity of Algorithm \ref{youngtocharpolyalg} (run on all $|\la| \leq n$) is bounded by
			$$O\bigg(\sum_{|\la| \leq n} \sum_{|\rho| \leq |\la|} \sum_\mu 1\bigg) \leq O\bigg(\sum_{|\la| \leq n} \sum_{|\rho| \leq n} \sum_{|\mu| \leq n} 1\bigg) = O(A(n)^3).$$
			
			\item 
			Given that each of the power series $\Phi_\rho^\infty(z)$ have already been computed, Algorithm \ref{polystatalg} solely involves taking the sum of a collection of power series indexed by the partitions $\rho$ with $|\rho| \leq |\la|$, each containing $maxDegree$ terms. Thus, Algorithm \ref{polystatalg} has algorithmic complexity
			$$O\bigg(\sum_{|\la| \leq n} \sum_{|\rho| \leq |\la|} maxDegree\bigg) \leq O(A(n)^2 maxDegree).$$
		\end{enumerate} 
		Therefore, the overall algorithmic complexity for all $|\la| \leq n$ is
		$$O(A(n)n^2maxDegree + A(n)^2n + A(n)^3 + A(n)^2 maxDegree)$$
		$$=\boxed{O(A(n)^3 + A(n)^2maxDegree)}.$$
		In practice it is faster.
	\end{enumerate}

	\printbibliography

\end{document}